\documentclass[12pt,leqno]{amsart}
\usepackage{amssymb,amsthm,amsmath,latexsym}
\newtheorem{theorem}{\sc Theorem}[section]
\newtheorem{thm}[theorem]{\sc Theorem}
\newtheorem{lem}[theorem]{\sc Lemma}
\newtheorem{prop}[theorem]{\sc Proposition}
\newtheorem{cor}[theorem]{\sc Corollary}

\newtheorem{ex}[theorem]{\sc Example}

 \DeclareMathOperator{\PSL}{PSL}
 \DeclareMathOperator{\FC}{FC}

 \DeclareMathOperator{\ASL}{ASL}
 \DeclareMathOperator{\SL}{SL}

 \DeclareMathOperator{\Aut}{Aut} 
 \DeclareMathOperator{\Tor}{Tor}  
 \DeclareMathOperator{\cent}{Z}

\title[Automorphism Orbits]{Virtually nilpotent groups with finitely many orbits under automorphisms}
\author[Bastos]{Raimundo  Bastos}
\address{Departamento de Matem\'atica, Campus Universit\'{a}rio Darcy Ribeiro,  Universidade de Bras\'ilia,
Brasilia-DF, 70910-900 Brazil}
\author[Dantas]{Alex C. Dantas}
\author[de Melo]{Emerson de Melo}
\email{(Bastos) bastos@mat.unb.br; (Dantas) alexdantas@unb.br; (de Melo) emerson@mat.unb.br}
\thanks{The authors were supported by DPI/UnB and FAPDF-Brazil.}
\begin{document}
\maketitle

\begin{abstract}
Let $G$ be a group. The orbits of the natural action of $\Aut(G)$ on $G$ are called ``automorphism orbits'' of $G$, and the number of automorphism orbits of $G$ is denoted by $\omega(G)$. Let $G$ be a virtually nilpotent group such that $\omega(G)< \infty$.  We prove that $G = K \rtimes H$ where $H$ is a torsion subgroup and $K$ is a torsion-free nilpotent radicable characteristic subgroup of $G$. Moreover, we prove that $G^{'}= D \times \Tor(G^{'})$ where $D$ is a torsion-free nilpotent radicable characteristic subgroup.  In particular, if the maximum normal torsion subgroup $\tau(G)$ of  $G$ is trivial, then $G^{'}$ is nilpotent.
\end{abstract}

\maketitle

\section{Introduction}

Let $G$ be a group. The orbits of the natural action of $\Aut(G)$ on $G$ are called ``automorphism orbits'' of $G$, and the number of automorphism orbits of $G$ is denoted by $\omega(G)$. Observe that automorphism orbits are unions of conjugacy classes and hence they give an example of \textit{fusion} in the holomorph group $G \rtimes \Aut(G)$, a well-known concept established in the literature (see for instance \cite[Chapter 7]{gorenstein}). It is interesting to ask what can we
say about ``$G$'' only knowing $\omega(G)$. It is obvious that $\omega(G)=1$ if and only if $G=\{1\}$, and it is well-known that if $G$ is a finite group then $\omega(G) = 2$ if and only if $G$ is elementary abelian. In \cite{LM}, T.\,J. Laffey and D. MacHale proved that if $G$ is a finite non-soluble group with $\omega(G) \leqslant 4$, then $G$ is isomorphic to $\PSL(2,\mathbb{F}_4)$. M. Stroppel in \cite{S1} has shown that the only finite nonabelian simple groups $G$ with $\omega(G) \leqslant 5$ are the groups $\PSL(2,\mathbb{F}_q)$ with $q \in \{4,7,8,9\}$ (see also \cite{K04} for finite simple groups with $\omega(G) \leqslant 17$). In \cite{BDG}, it was proved that if $G$ is a finite non-soluble group with $\omega(G) \leqslant 6$, then $G$ is isomorphic to one of $\PSL(2,\mathbb{F}_q)$ with $q \in \{4,7,8,9\}$, $\PSL(3,\mathbb{F}_4)$ or $\ASL(2,\mathbb{F}_4)$ (answering a question of M. Stroppel \cite[Problem 2.5]{S1}). Here $\ASL(2,\mathbb{F}_4)$ is the affine group $\mathbb{F}_4^2 \rtimes \SL(2,\mathbb{F}_4)$ where $\SL(2,\mathbb{F}_4)$ acts naturally on $\mathbb{F}_4^2$. For more details concerning automorphism orbits of finite groups see \cite{BDG,Bors,K04,LM,Z}.

Some aspects of automorphism orbits are also investigated for infinite groups. In \cite{MS1}, H. M\"aurer and M. Stroppel classified the groups with a nontrivial characteristic subgroup and with three orbits by automorphisms (see also \cite[Theorem B]{BDdM}). M. Schwachh\"ofer and M. Stroppel have shown that if $G$ is an abelian group with finitely many automorphism orbits, then $G = D \oplus \Tor(G)$, where $D$ is a torsion-free divisible characteristic subgroup of $G$ and $\Tor(G)$ is the set of all torsion elements in $G$ \cite[Lemma 1.1]{S2}. In \cite{BD18}, it was proved that if $G$ is a $\FC$-group with finitely many automorphism orbits, then the derived subgroup $G'$ is finite and $G$ admits a decomposition $G = A \times \Tor(G)$, where $A$ is a divisible characteristic subgroup of $\cent(G)$ (cf. \cite[Theorem A]{BD18}). 

The semidirect product allows us to present interesting examples of groups with finitely many automorphism orbits (see Section \ref{sec:examples}, below). The next result provides a criterion to a semidirect product $G = A \rtimes B$ has only finitely many automorphism orbits, where $A$ is an abelian group and $B$ is a finite group. 

\begin{thm}\label{pro1}
Let $A$ be an abelian group and $B$ a finite subgroup of $\Aut(A)$.  Let $G=A \rtimes B$ be the semidirect product of $A$ and $B$ and assume that $A$ is a characteristic subgroup of $G$. Then $\omega(G)< \infty$ if and only if $A$ has finite automorphism orbits under the action of $C_{\Aut(A)}(B)$.
\end{thm}

We recall that a group possesses a certain property virtually if it has a subgroup of finite index with that property. In \cite{BDdM}, the authors prove that if $G$ is a soluble group of finite rank with $\omega(G)< \infty$, then $G$ contains a torsion-free characteristic nilpotent subgroup $K$ such that $G = K \rtimes H$, where $H$ is a finite group (cf. \cite[Theorem A]{BDdM}). The next result can be viewed as a generalization of the above mentioned results from \cite{BD18}, \cite{BDdM} and \cite{S2}. 

\begin{thm}\label{thm:linear}
Let $G$ be a virtually nilpotent group with $\omega(G)<\infty$. Then $G$ has a torsion-free radicable nilpotent subgroup $K$ and a torsion subgroup $H$ such that $G=K \rtimes H$. Moreover, the derived subgroup $G^{'}= D \times \Tor(G^{'})$, where $D$ is a torsion-free nilpotent radicable characteristic subgroup.
\end{thm}

We denote by $\tau(G)$ the maximum normal torsion subgroup of a group $G$. Note that in Theorem \ref{thm:linear} if $\tau(G)=1$, then the derived subgroup $G'$ is nilpotent. On the other hand, $H$ need not be soluble, for instance, if $H$ is any finite group and $D$ is a torsion-free divisible abelian subgroup, then the group $G = H \times D$ has finitely many orbits under automorphisms (cf. \cite[Lemma 1.1(4)]{S2}).

Now, the following result is an immediate consequence of \cite[Theorem A]{BDdM} and Theorem \ref{thm:linear}. 

\begin{cor} \label{cor.finite.rank}
Let $G$ be a virtually soluble group of finite rank with $\omega(G)< \infty$. Then $G$ has a torsion-free radicable nilpotent subgroup $K$ and a finite subgroup $H$ such that $G=K \rtimes H$. Moreover, the derived subgroup $G^{'}= D \times \Tor(G^{'})$, where $D$ is a torsion-free nilpotent radicable characteristic subgroup.
\end{cor}

As mentioned before, if $A$ is a torsion-free abelian group with $\omega(A)<\infty$, then $A$ is a divisible abelian group. In particular, $A$ can be considered as a vector space over $\mathbb{Q}$. Using Maschke's theorem and Theorems \ref{pro1} and \ref{thm:linear} (see also Lemma \ref{abel}, below) we obtain the following interesting corollary.

\begin{cor}\label{pro2}
Let $A$ be a finite dimensional vector space over $\mathbb{Q}$ and $B$ a finite subgroup of $\Aut(A)$. Let $G=A \rtimes B$ be the semidirect product of $A$ and $B$. Then $\omega(G)< \infty$ if and only if $B$ is abelian. \end{cor}

It is worth to mention that in Corollary \ref{pro2} each element of $C_{\Aut(A)}(B)$ induces automorphisms of $A$ as $\mathbb{Q}B$-module. Then, Theorem \ref{pro1} says that the semidirect product $A \rtimes B$ has finitely many automorphism orbits if and only if $A$ has finitely many automorphism orbits as $\mathbb{Q}B$-module.  

\section{Preliminaries and Examples}  \label{sec:examples}

As usual, if $A$ is a group and $B$ is a subgroup of $\Aut(A)$, we denote by $C_{\Aut(A)}(B)$ the subgroup $\{c\in \Aut(A)  \mid  b^c=b, \  \forall b\in B \}$. We write $\omega_{C}(G)$ for the number of automorphism orbits of $G$ under the action of a subgroup $C$ of $\Aut(G)$. In what follows, $V$ is a vector space over $\mathbb{Q}$, $B$ is a finite group of $\Aut(V)$ and $C=C_{\Aut(V)}(B)$. We consider $V$ as a $\mathbb{Q}BC$-module. 

\begin{lem} \label{rem}
If $\omega_{C}(V) = n$, then $\omega_{C^{s}}(V^{s}) = n^{s}$. 
\end{lem}
\begin{proof}
It is enough to note that if $v_{1}, \cdots, v_{n}$ are representatives of the orbits of $V$ by action of $C$, then $\{(v_{i_{1}}, \cdots, v_{i_{s}}) \mid 1 \leqslant i_{1}, \cdots, i_{s} \leqslant n \}$ is a set of representatives of the orbits of $V^{s}$ by action of $C^{s}$.
\end{proof}

Assume that $|B|=s$ and let $\xi$ be a primitive $s$th root of unity. We extend the ground ring $\mathbb{Q}$ of $V$ by 
$\xi$ and denote by $\tilde{V}$ the vector space  $V\otimes_{\mathbb{Q}}\mathbb{Q}[\xi]$. 
The actions of $B$ and $C$ on $V$ extends naturally to $\tilde{V}$ by $(v\otimes k)^x=v^x\otimes k$. See \cite[Chapter 4]{Khukhro} for more details. If  $\phi(s)$ is Euler's function, we obtain that $$\tilde{V} =V\otimes_{\mathbb{Q}}\mathbb{Q}[\xi]= \bigoplus_{i=0}^{\phi(s)-1} (V\otimes \mathbb{Q}\xi^i).$$ Clearly  if $\omega_C(V) < \infty$, then $\omega_{C}(V\otimes \mathbb{Q}\xi^i)< \infty$ for any $i$. Therefore by Lemma \ref{rem} we conclude that $\omega(\tilde{V})< \infty$.

\begin{lem}\label{Lem1}
Let $W \leqslant V$ be an irreducible $C$-submodule. Assume that $\omega_{C}(V) < \infty$. Then $W^b = W$ for each $b \in B$.
\end{lem}
\begin{proof}
Note that both $W$ and $W^b$ are irreducible $C$-submodule of $V$. Suppose that $W^b\neq W$. Then $W\cap W^b=0$. Let $w\in W$ and $1\neq c\in C$. We will prove that $w+w^b$ and $w+(w^b)^c$ are not in the same orbit. In fact, if $(w+w^b)^x=w+(w^b)^c$ for any $x\in C$, then $w^x+(w^b)^x=w+(w^b)^c$ which implies that $w^x=w$. Thus, $(w+w^b)^x=w+w^b$. Therefore, for any $c\in C$ the elements $w+(w^b)^c$ are in different orbits. 
\end{proof}

\begin{lem}
Let $W \leqslant \tilde{V}$ be an irreducible $C$-submodule, where $\tilde{V}=V\otimes_{\mathbb{Q}}\mathbb{Q}[\xi]$. If $\omega_{C}(\tilde{V}) < \infty$, then any element $b\in B$ acts as scalar on $W$. 
\end{lem}
\begin{proof}
By Lemma \ref{Lem1} we have that $W^b=W$ for any $b\in B$. On the other hand, for any $b\in B$ the subspace $\{v\in V \mid v^b=\lambda v, \ \ \lambda \in \mathbb{Q}[\xi] \}$ is a $C$-submodule of $\tilde{V}$. Then any element $b\in B$ acts as scalar on $W$, since $\mathbb{Q}[\xi]$ is a splitting field for $B$.
\end{proof}

\begin{lem}\label{abel}
If $\omega_C(V)<\infty$, then $B$ is abelian.
\end{lem}
\begin{proof}

Since $V$ and $\tilde{V}$ have finitely many orbits under the action of $C$, without loss of generality it will be assumed that $V$ is a vector space which the ground field is a finite extension of $\mathbb{Q}$ and contains $\xi$.

For any $v\in V$ it is easy to see that $v$ is contained in a finite dimensional $B$-invariant subspace of $V$. So, although $V$ is not of finite dimension, for any $b\in B$ we have that $b$ is diagonalizable and then $$V=\bigoplus_{\lambda \in \langle \xi \rangle} V_{\lambda}$$ where $V_{\lambda}$ is the eigenspace of $b$ associated with $\lambda$. Since each subspace $V_{\lambda}$ is a $C$-submodule, we obtain that $V_{\lambda}^{y}=V_{\lambda}$ for any $y \in B$ by Lemma \ref{Lem1}. Now it became clear that $V$ is a direct sum of ``simultaneous eigenspaces'', that is, $V$ is a direct sum of subspace of the form  $\{v\in V \mid v^{b}=\lambda(b)v \ \textrm{for each } b\in B\}$ where $\lambda(b)$ is a eigenvalue of $b\in B$. Therefore $B'=1$, since $B'$ acts trivially on simultaneous eigenvectors, which completes the proof. 
\end{proof}

The remainder of this section is devoted to present semidirect products with exactly three orbits under automorphisms. Recall that a group $G$ is called almost homogeneous if $\omega(G) \leqslant 3$. The next example compile some almost homogeneous groups which appeared in the papers \cite{BDdM,LM,MS1}.   

\begin{ex}\label{ex:almost-homogeneous-groups}
{ \ }
\begin{enumerate}
\item In \cite{LM}, T.\,J. Laffey and D. MacHale characterize finite groups $G$ that are not $p$-groups and have the property that $\omega(G) = 3$. According to this characterization, such a group has order of the form $pq^n$ where $p$ and $q$ are distinct
primes. Furthermore, the Sylow $q$-subgroup $Q$ is elementary abelian and the Sylow $p$-subgroup acts fixed-point-freely on $Q$.  
\item (H. M\"aurer, M. Stroppel \cite[Lemma 2.6 and Proposition 2.7]{MS1}) Let $p$ be a prime. Let $\mathbb{F}$ be a field, and assume that the cyclotomic polynomial $\Phi_p$ given by $\Phi_{p} = (x^{p}-1)/(x-1)$ has $p-1$ different roots in $\mathbb{F}$, and that $\Aut(\mathbb{F})$ acts transitively on the set of these roots. For any vector space $V$ over $\mathbb{F}$, define the semidirect product $V \rtimes \Omega$, where $\Omega=\{f \in \mathbb{F} \mid f^p=1\}$ acts on $V$ by scalar multiplication. Then $\omega(V \rtimes \Omega)=3$.
\item In \cite{BDdM}, the authors prove that a soluble mixed order group $G$ has $\omega(G)=3$ if and only if $G=A \rtimes H$ where $|H|=p$ for some prime $p$, $H$ acts fixed-point-freely on $A$ and $A=\mathbb{Q}^n$ for some positive integer $n$.
\end{enumerate}
\end{ex}

The next example shows that there are nonnilpotent torsion-free groups of infinite rank with exactly three orbits under automorphisms. 

Let $c$ be a real transcendental number. Define $D=\langle c^{a} \mid a \in \mathbb{Q}^{+} \rangle$. Note that $D$ is a subgroup of  $\mathbb{R}^{\times}$ isomorphic to the additive group $\mathbb{Q}^+$. Now, let $S$ be the set of nontrivial elements of $D$, that is, $S=\{c^{a} \mid a \in \mathbb{Q}^{\times}\}$ and $\mathbb{K}=\mathbb{Q}(S)$ the transcendental  field extension generated by $S$. It was proved in \cite[Example 2.2]{Robert} that the automorphism groups of the field $\mathbb{K}$ is isomorphic to $\mathbb{Q}^{\times}$ and it is generated by automorphism mapping $c$ to $c^{a}$ where $a \in \mathbb{Q}^{\times}$. In particular, under this automorphisms $D$ has two orbits.

\begin{ex} \label{ex.transcendental}
Let 
$$G = \left\{\left[ \begin{array}{cc}
                        d & 0 \\
                        k & 1 
                        \end{array}  \right] \mid d\in D, k\in \mathbb{K} \right\} \simeq \mathbb{K}\rtimes D.$$
Then $\omega(G) = 3$. 
\end{ex}
\begin{proof}
Let $C=\left\{\left[ \begin{array}{cc}
                        k' & 0 \\
                        0 & 1
                        \end{array}  \right] \mid k'\in  \mathbb{K}^{\times} \right\}$. Then $C$ acts by automorphisms on $G$ and under this action the subgroup $$N_1=\left\{\left[ \begin{array}{cc}
                        1 & 0 \\
                        k & 1
                        \end{array}  \right] \mid k\in  \mathbb{K} \right\}$$ has exactly two orbits. Moreover, $C$ acts trivially on $$N_2=\left\{\left[ \begin{array}{cc}
                        d & 0 \\
                        0 & 1
                        \end{array}  \right] \mid d\in  D \right\}.$$ On the other hand, the automorphisms of $\mathbb{K}$ acts on $G$ in such a way that $N_2$ has two orbits. Then $G$ has three automorphism orbits. 
\end{proof}

It is worth to mention that infinite groups with small number of automorphism orbits need not be soluble. For instance, G. Higman, B. Neumann and H. Neumann have constructed a torsion-free nonabelian simple group with exactly $2$ automorphism orbits. Recently, D. Osin, in \cite{Osin}, proved that any countable torsion-free group can be embedded into a $2$-generated group $M$ with exactly $2$ conjugacy classes. In particular, $\omega(M)=2$. 

\section{Proof of Theorem \ref{pro1}}

In this section we prove Theorem \ref{pro1} and Corollary \ref{pro2}.


\begin{proof}[Proof of Theorem \ref{pro1}]

Recall that $A$ is an abelian group and $B$ a finite subgroup of $\Aut(A)$. Assume that $A$ is a characteristic subgroup of $G = A \rtimes B$. We need to prove that $\omega(G)< \infty$ if and only if $A$ has finite automorphism orbits under the action of $C_{\Aut(A)}(B)$.

Suppose that $\omega(G)$ is finite. Taking into account that $A$ is a characteristic subgroup of $G$ we can consider $\Aut(G)$ acting on $G/A$ and then we conclude that there exists a subgroup $X \leqslant \Aut(G)$ of finite index which acts trivially on $G/A$. In other words, for any $b\in B$ the cosets $Ab$ are invariant by all elements of $X$. Note that $G$ has finite orbits under the action of $X$ since $X$ has finite index.

Now, for each $\alpha \in X$ we  define $\bar{\alpha} \in C_{\Aut(A)}(B)$ by
$$(ab)^{\bar{\alpha}} = a^{\alpha} b, \text{ for any } a\in A \text { and } b\in B.$$

In fact, since $A$ is abelian and $b^{\alpha}=cb$ for some $c\in A$ we have
$$(a^{b})^{\bar{\alpha}} = (a^b)^{\alpha} = (a^{\alpha})^{b^{\alpha}} = (a^{\alpha})^{c b} = (a^{\alpha})^{b} = (a^{\bar{\alpha}})^{b^{\bar{\alpha}}},$$
and then ${\bar{\alpha}}$ is an automorphism of $G$ such that $b^{\bar{\alpha}}=b$. Moreover, note that for each $\alpha \in X$ the automorphism $\bar{\alpha}$ acts in the same way on $A$. Then number of orbits in $A$ under the action of $X$ and $C_{\Aut(A)}(B)\cap X$ is the same. 

Now, assume that $A$ has finite orbits by the action of $C_{\Aut(A)}(B)$. For any $c\in C_{\Aut(A)}(B)$, $a \in A$ and $b \in B$ we have that $(ab)^{c} = a^{c} b$. Therefore $G=A \rtimes B$ has at most $|B|\omega_{C}(A)$ automorphism orbits.
\end{proof}

The next lemma can be found in \cite[Theorem 3.2.2]{gorenstein}.

\begin{lem}\label{represAbel}
If a finite abelian group $H$ possesses a faithful irreducible representation on a vector space $V$, then $H$ is cyclic.
\end{lem}

We are now in a position to prove Corollary \ref{pro2}.

\begin{proof}[Proof of Corollary \ref{pro2}]
Recall that $A$ is a finite dimensional vector space over $\mathbb{Q}$ and $B$ a finite subgroup of $\Aut(A)$. We need to show that $\omega(A \rtimes B)< \infty$ if and only if $B$ is abelian.

If $\omega(G)< \infty$, then we can use Lemma \ref{abel}
 to conclude that $B$ is abelian.
 
 Assume that $G=A \rtimes B$ where $B$ is an abelian group. We consider $A$ as a vector space over $\mathbb{Q}$. Then Maschke's Theorem \cite[3.3.1]{gorenstein} gives that $A$ is a completely reducible $B$-module. Hence $$A=V_1\oplus \cdots \oplus V_t $$ where $V_i$ is an irreducible $B$-submodule.   In particular, we have that there exists a $B$-invariant subspace $U$ such that $A=V_j\oplus U$ for any $j\in \{1,\cdots,t\}$.
 
 Let $V_j$ be a component of $A$ in the above decomposition. By Lemma \ref{represAbel} if $K$ is the kernel of the action of $B$ on $V_j$, then $B/K$ is cyclic. Let $x$ be a generator of $B/K$ and $d$ the degree of the minimal polynomial of $x$.  Let $v$ and $w$ be nontrivial elements of $V_j$. Then $\{v, \cdots, v^{x^d}\}$ and $\{w, \cdots, w^{x^d}\}$ are bases of $V_j$. Now,  a routine calculation shows that the map $\alpha$ defined by $(v^{x^j})^{\alpha}=w^{x^j}$, $u^{\alpha}=u$ ($u\in U$) and $b^{\alpha}=b$ for any $b\in B$ extends to an automorphism of $G$. Indeed, the automorphism $\alpha$ commute with $x$ since $(v^{x^j})^{\alpha}=(w^{x^{j}})=(v^{\alpha})^{x^j}$. Hence, for any component $V_j$ and any two elements $v,w\in V_j$ there is $\alpha \in C_{\Aut(A)}(B)$ such that $v^{\alpha}=w$. Therefore $A$ has finitely many automorphism orbits under the action of $C_{\Aut(G)}(B)$ and by Theorem \ref{pro1} it is sufficient to conclude that $\omega(G)< \infty$.
\end{proof}

\section{Proof of Theorem \ref{thm:linear}}

Let $G$ be a group and $H$ a subgroup of $\Aut(G)$. We say that $H$ stabilizes a normal $H$-invariant series $$G=G_1>G_2>\cdots >G_t>G_{t+1}=1$$ if $A$ acts trivially on each factor $G_{i-1}/G_i$.

\begin{lem}\label{stab}
Let $G$ be a torsion-free group and $H$ a finite subgroup of $\Aut(G)$. If a subgroup $N$ of $H$ stabilizes a normal $H$-invariant series $$G=G_1>G_2>\cdots >G_t>G_{t+1}=1,$$ then $N=1$.
\end{lem}
\begin{proof}
For any $x\in N$ and $g\in G_{t-1}$ we have that $g^x=gy$ for some $y\in G_t$. Moreover, we have that $g^{x^i}=gy^{i}$ for any $i$. Hence, if $|x|=n$, then $g=gy^n$, whence $y=1$ and therefore $g^x=g$. Now, by induction on $t$ it is easy to see that $g^x=g$ for any $g\in G$.
\end{proof}

Recall that a group $G$ is said to be radicable if each element is an $n$th power for every positive integer $n$. If $G$ is abelian it is also called divisible. Note that if $G$ is a nilpotent group such that the factor $\gamma_i(G)/\gamma_{i+1}(G)$ is divisible for any $i$, then $G$ is radicable. The next two results were proved in \cite{BDdM}.

\begin{lem}\label{lem.abelian.sub}
Let $A$ be a divisible normal abelian subgroup of finite index of a group $G$. Then there exists a subgroup $H$ of $G$ such that $G = A \rtimes H$.
\end{lem}

\begin{prop}
\label{lem.soluble.sub} Let $n$ be a positive integer. Let $G$ be a group such that $\omega(G)< \infty$ and $K$ a torsion-free characteristic subgroup of $G$ with finite index $n$. If $K$ is soluble, then there exists a subgroup $H$ of $G$ such that $G = K \rtimes H$. 
\end{prop}

M. Schwachh\"ofer and M. Stroppel in \cite{S2}, have shown that if $G$ is an abelian group with finitely many automorphism orbits, then $G = D \oplus \Tor(G)$, where $D$ is a torsion-free divisible characteristic subgroup of $G$ (see also \cite[Theorem A]{BD18}). Now, we extends this result to nilpotent groups. 

\begin{prop}
\label{lem.nilpotent} Let $G$ be a nilpotent group such that $\omega(G)< \infty$. Then $G=K \times \Tor(G)$ where  $K$ is torsion-free radicable subgroup.
\end{prop}
\begin{proof}
We argue by induction on the nilpotency class of $G$. 

Assume that $G$ is abelian. By Schwachh\"ofer-Stroppel's result \cite{S2}, $G = D \oplus T$, where $D$ is characteristic torsion free divisible subgroup and $T$ is the torsion subgroup of $G$.

Now, we assume that $G$ is non-abelian. Arguing as in the previous paragraph, we deduce that  $\cent(G) = D_1 \oplus T_1$, where $D_1$ is a characteristic torsion-free divisible subgroup and $T_1$ is the torsion subgroup of $\cent(G)$. By induction $G/\cent(G)$  has the desired decomposition. More precisely, $\cent(G) = D_1 \oplus T_1$ and $G/\cent(G) = \bar{A} \times \Tor(\bar{G})$ where $\bar{A}$ is torsion-free radicable subgroup. Let $A$ be the inverse image of $\bar{A}$. It is clear that $\Tor(A)=T_1$. Set $K=A^e$, where $e=\exp(T_1)$. Then $K$ is torsion-free radicable subgroup and $G/K$ is a torsion group. Since $K$ and $\Tor(G)$ are a normal subgroups, we have that $\langle K, \Tor(G) \rangle=K \times \Tor(G)$. Then, it is suficient to prove that $G=\langle K, \Tor(G) \rangle$. Let $x \in G \setminus K$. Thus, the subgroup $\langle K, x\rangle$ has a torsion-free characteristic subgroup of finite index. We will use induction on the nilpotency class of $K$ and Lemma \ref{lem.abelian.sub} to prove that $\langle K, x\rangle= K \times H$ where $H$ is a finite group. In fact, if $K$ is abelian, then by Lemma \ref{lem.abelian.sub} $\langle K, x\rangle= K \times H$ where $H$ is a finite group. If $K$ is not abelian, then by induction the torsion subgroup of $\langle K, x\rangle/\cent(K)$ is finite. Let $N$ be the inverse image of the torsion subgroup of $\langle K, x\rangle/\cent(K)$. Then $\cent(K)$ has finite index in $N$ and by Lemma \ref{lem.abelian.sub} $N=\cent(K) \times H$ where $H$ is a finite group. Therefore $\langle K, x\rangle= K \times \Tor(G) $, which completes the proof.
\end{proof}

Let us now prove Theorem \ref{thm:linear}.

\begin{proof}[Proof of Theorem \ref{thm:linear}]
Recall that $G$ is a virtually nilpotent group with $\omega(G)<\infty$. We need to prove that  $G$ has a torsion-free radicable nilpotent subgroup $K$ and a torsion subgroup $H$ such that $G=K \rtimes H$ and the derived subgroup $G^{'}= K \times \Tor(G^{'})$.

As $G$ is virtually nilpotent we have that $G$ contains a nilpotent characteristic subgroup $N$ of finite index. It is well-known that $\Tor(N)$ is a characteristic subgroup of $N$ and then normal in $G$. First assume that $\Tor(N)=1$. By Proposition  \ref{lem.soluble.sub}  $G = K \rtimes H$, where $K=N$ is torsion-free and $H$ is a finite subgroup.

Consider an unrefinable central series of $K$ of characteristic subgroups $$K=K_1>K_2>\cdots >K_t>K_{t+1}=1.$$

Note that every factor $V_i=K_i/K_{i+1}$ is an abelian group with finitely many orbits and then $V_i$ can be considered as a vector space over $\mathbb{Q}$.

Now, taking into account that $K$ is a characteristic subgroup of $G$ we can consider $\Aut(G)$ acting on $G/K$ and then we conclude that there exists a subgroup $X \leqslant \Aut(G)$ of finite index which acts trivially on $G/A$. In other words, for any $h\in H$ the cosets $Kh$ are invariant by all elements of $X$. Note that $G$ has finite orbits under the action of $X$ since $X$ has finite index.

Now, for each $\alpha \in X$ we have that $h^{\alpha}=kh$ for some $k\in K$. Then for any $v\in V_i$ we have 
$$(v^h)^{\alpha} = (v^{\alpha})^{h^{\alpha}} = (v^{\alpha})^{kh} = (v^{\alpha})^{h} ,$$
since $[v^{\alpha},k]\in K_{i+1}$. In other words, the action of $H$ and $X$ commute in $V$. Now by Lemma \ref{abel} $H^{'}$ acts trivially on $V_i$ for any $i$. Thus, $H^{'}$ acts trivially on $K$ by Lemma \ref{stab} and then $H^{'} \leqslant \tau(G)$.

Now, assume that $\Tor(N)\neq 1$. By Proposition \ref{lem.nilpotent}, $N=K \times \Tor(N)$ where $K$ is a torsion-free subgroup. Consider $\bar{G}=G/\Tor(N)$. Clearly the inverse image $H$ of $\bar{H}$ is a torsion subgroup and $G = K \rtimes H$ where $K$ is a torsion-free nilpotent subgroup of $G$. On the other hand, $\Tor(N)$ and $H'$ are contained in $\tau(G)$ and $\tau(G)$ commute with $K$. Therefore  $G' \leqslant K\rtimes H'=K\times H'$ which completes the proof.
\end{proof}

\end{document}